\documentclass[12pt]{amsart}
\usepackage{amsmath,amssymb,accents,amsthm}
\usepackage[latin1]{inputenc}
\usepackage{color}
\usepackage[all]{xypic}
\allowdisplaybreaks[4]
\newtheorem{theorem}{Theorem}[section]
\newtheorem{proposition}[theorem]{Proposition}
\newtheorem{corollary}[theorem]{Corollary}
\newtheorem{lemma}[theorem]{Lemma}

\theoremstyle{definition}
\newtheorem{example}[theorem]{Example}
\newtheorem{definition}[theorem]{Definition}
\newtheorem{remark}[theorem]{Remark}
\numberwithin{equation}{section}
\renewcommand\({\Big(}
\renewcommand\){\Big)}

\newcommand\<{\langle}
\renewcommand\>{\rangle}
\newcommand\er{\eqref}
\newcommand\bi{\begin{itemize}}
\newcommand\ei{\end{itemize}}
\newcommand\bm{\begin{pmatrix}}
\renewcommand\em{\end{pmatrix}}
\newcommand\be{\begin{equation}\label}
\newcommand\ee{\end{equation}}

\newcommand\I{\int\limits}
\renewcommand\S{\sum\limits}
\renewcommand\P{\prod\limits}

\newcommand\dl{\partial}
\newcommand\oc{\circ}
\newcommand\op{\oplus}
\newcommand\al{\approx}
\newcommand\el{\ell}
\newcommand\la{\alpha}
\newcommand\lb{\beta}

\newcommand\ld{\delta}
\newcommand\lD{\Delta}
\newcommand\Le{\varepsilon}
\newcommand\lh{\eta}
\newcommand\li{\iota}

\renewcommand\ll{\lambda}

\renewcommand\ln{\nu}

\newcommand\lt{\vartheta}

\newcommand\Lt{\tau}
\newcommand\lo{\omega}
\newcommand\lO{\Omega}
\newcommand\lx{\xi}

\newcommand\lz{\zeta}

\newcommand\iu{\cup}
\newcommand\ui{\cap}

\newcommand\ic{\subset}

\newcommand\xx{\times}

\newcommand\oo{\infty}
\newcommand\OO{\emptyset}

\newcommand\sm{\setminus}

\newcommand\EL{\mathcal E}

\newcommand\GL{\mathcal G}
\newcommand\HL{\mathcal H}

\newcommand\PL{\mathcal P}

\newcommand\VL{\mathcal V}

\newcommand\Bl{\mathbf B}
\newcommand\Cl{\mathbf C}

\newcommand\Ol{\mathbf O}

\newcommand\Tl{\mathbf T}

\renewcommand\k{{\boldsymbol k}}

\newcommand\m{{\boldsymbol m}}

\newcommand\bb[4]{\begin{pmatrix}{#1}&{#2}\\{#3}&{#4}\end{pmatrix}}
\newcommand\ba[2]{\begin{pmatrix}{#1}\\{#2}\end{pmatrix}}
\newcommand\ca[3]{\begin{pmatrix}{#1}\\{#2}\\{#3}\end{pmatrix}}
\newcommand\da[4]{\left(\begin{array}{c}{#1}\\{#2}\\{#3}\\{#4}\end{array}\right)}
\newcommand\cc[9]{\begin{pmatrix}{#1}&{#2}&{#3}\\{#4}&{#5}&{#6}\\{#7}&{#8}&{#9}\end{pmatrix}}

\newcommand\h[1]{{\hat{#1}}}

\renewcommand\t[1]{{\tilde{#1}}}

\renewcommand\o[1]{{\overline{#1}}}
\renewcommand\u[1]{{\underline{#1}}}
\renewcommand\v[1]{{\vec{#1}}}
\newcommand\F[1]{\sqrt{#1}}
\newcommand\f[2]{\frac{#1}{#2}}
\begin{document}

\setcounter{section}{-1}

\title{Holomorphic isometries from the unit ball into symmetric domains}
\author{Harald Upmeier, Kai Wang and Genkai Zhang}

\address{ Fachbereich Mathematik, Universit\"at Marburg, Marburg, 35032, Germany}
\email{upmeier@mathematik.uni-marburg.de}

\address{School of Mathematical Sciences, Fudan University, Shanghai, 200433, P. R. China}
\email{kwang@fudan.edu.cn}

\address{Department of Mathematical Sciences, Chalmers and Gothenburg University, Gothenburg 41296, Sweden}
\email{genkai@chalmers.se}

 \subjclass[2010]{32M, 32H, 17C}
\keywords{Bounded symmetric domains, holomorphic isometries, Jordan triples,
second fundamental forms, rational mappings}
\thanks{ K. Wang  was partially supported by NSFC (11271075,11420101001), the Humboldt Foundation, Laboratory of Mathematics for Nonlinear Science at Fudan University, and the Magnusson Foundation of Royal Swedish Academy of Sciences.
Research by G. Zhang partially supported by the Swedish Science Council (VR)}

\maketitle

\begin{abstract} We construct rational isometric holomorphic embeddings of the unit ball into higher rank symmetric domains $D$, first discovered by Mok, in an explicit way using Jordan triple systems, and we classify all isometric embeddings into tube domains of rank 2. For symmetric domains of arbitrary rank, including the exceptional domains of dimension 16 and 27, respectively, we characterize the Mok type mapping in terms of a vanishing condition on the second fundamental form of the image of  $F$ in $D$.
 \end{abstract}

\section{Introduction} The study of holomorphic isometries, and more generally proper holomorphic maps, between general K\"ahler manifolds is a classical topic in complex geometry, starting with E. Calabi; see \cite{Calabi, Tsai-JDG, Mok-Ng, Ng-PAMS, Chan, Mok-Ng-meas-pre, Ng-MZ, Huang-Yuan, Fang-Huang-Xiao}. In a series of papers \cite{Mok-JEMS, Mok-ALM, Mok-unitdisc, Mok-2nd-form} N. Mok initiated the study of holomorphic isometric mappings from the unit ball $\Bl=\Bl_d$ in $\Cl^d$ into bounded symmetric domains. Besides the totally geodesic embeddings, Mok \cite{Mok-non-geo-iso} constructed a new class of non-totally geodesic rational holomorphic isometric embeddings of $\Bl$ into a general bounded symmetric domain $D$ of rank $r\ge 2$. In the present paper we shall give an elementary and explicit construction of these Mok type mappings via the Jordan algebraic description of bounded symmetric domains \cite{L1}. Moreover, as a main result of this paper, we obtain a differential-geometric characterization of Mok type embeddings into bounded symmetric domains of arbitrary rank, including the exceptional domains. It turns out that the Jordan theoretic description is essential both for the construction of holomorphic isometries and for proving the characterization.
 
For the Lie ball $D$ of dimension $d+1$, $d>1,$ we also construct a class of irrational holomorphic isometric embeddings and prove that all holomorphic isometries from $\Bl$ into $D$ are either Mok type mappings or these irrational mappings up to reparametrization by automorphisms of $\Bl$ and $D.$ After a preliminary version of this paper was finished we noticed the preprints \cite{XY,Chan-Mok-2016}, where such results are also obtained. However their methods are different from ours. In particular the characterization of Mok type mappings
for non-tube domains
or higher rank domains is not treated there.

The paper is organized as follows. In Section 1 we recall some background on Jordan triple systems and their relation to bounded symmetric domains. In Section 2 we realize the Mok type rational mappings in terms of the Jordan theoretic 'quasi-inverse'. The differential-geometric characterization of Mok embeddings mentioned above is proved in Section 3. Section 4 contains the classification of Bergman isometries from the unit ball $\Bl_d$ into the Lie ball in $\Cl^{d+1}.$ In the final Section 5 we present an alternative approach to two results in \cite{Mok-JEMS}, concerning the extension of local holomorphic isometries and an explicit description of the graph of a holomorphic isometry.

We would like to thank Ngaiming Mok for explaining his work and Xiaojun Huang for sending us the papers \cite{Fang-Huang-Xiao, Huang-Yuan}. We also thank the referee for valuable comments.

\section{Bounded symmetric domains and isometric embeddings}
We briefly recall some known facts on bounded symmetric domains and their Jordan theoretic description. See \cite{FK,L1,U} and references therein for general background. Let $D$ be an irreducible bounded symmetric domain in $Z=\Cl^m$ of rank $r$. Let $G$ be the  identity component of the group of all biholomorphic mappings of $D$. Then $D=G/K$ is a realization of the Hermitian symmetric space $G/K$, where $K=\{g\in G; g\cdot 0=0\}$ is the isotropy subgroup of $0\in D$.  The group $G$ has a complexification $G^\Cl$ acting as rational mappings on $Z$. We denote  by $K^\Cl$ the corresponding complexification of $K$. It is well-known \cite{L1,U} that $Z$ can be equipped with the structure of a {\bf hermitian Jordan triple} and $D$ can then be realized as the 'spectral' unit ball in $Z.$ We denote the Jordan triple product by $(u,v,w)\mapsto\{u;v;w\}$,
which is $\mathbb C$ linear in $u,w$ and 
$\bar{\mathbb C}$-linear
in $v$,  and put 
$$Q_uv:=\f12\{u;v;u\},$$
$$D(u,v)w:=\{u;v;w\},$$
\be{1.14}B(u,v)w=w-\{u;v;w\}+Q_uQ_vw.\ee 
An element $c\in Z$ is called a {\bf tripotent} if $Q_cc=c.$ The endomorphism $B(u,v)\in\ End(Z)$ is called the {\bf Bergman operator}. We also write $B(u,v)=B_{u,v}.$ The covariance property
\be{1.8}B(h^{-1}u,h^*v)=h^{-1}\,B(u,v)\,h\ee
is known to hold for all $h\in K$, and it holds further for all
$h\in K^\Cl,$ since this relation is holomorphic in $h$ and holds for $h\in K,$ where $h^*=h^{-1}.$ It is well-known \cite{L1} that the {\bf Bergman metric} $ds^2_D$  at $z\in D$ is given by
$$ds^2_z(u, v)=\<u|v\>_z=\<B_{z,z}^{-1}u|v\>$$
where $\<u|v\>$ is the $K$-invariant inner product on $Z,$ normalized by $\<c|c\>=1$ for each minimal tripotent $c\in Z.$ The {\bf Bergman kernel} for the Bergman space $L_a^2(D)$, with respect to normalized Lebesgue measure on $D,$ has the form
\be{1.9}K^D(z,w)=\det B(z,w)^{-1}=\lD(z,w)^{-p}.\ee
Here $\lD:Z\xx\o Z\to\Cl$ is a sesqui-holomorphic polynomial called the {\bf quasi-determinant} (it is a denominator of the 'quasi-inverse' defined below.) The numerical invariant $p=2+a(r-1)+b$ is called the {\bf genus} of $D,$ and $a,b$ are the characteristic multiplicities of $Z.$ Then $\dim Z=r+\f a2(r-1)r+rb.$

\begin{example} For $r\le s,$ the matrix space $Z=\Cl^{r\xx s}$ is a hermitian Jordan triple with triple product
$$\{u;v;w\}=uv^*w+wv^*u.$$ 
The associated bounded symmetric domain is the {\bf matrix unit ball} 
$$D=\{z\in \Cl^{r\xx s}:\ 1-zz^*>0\}.$$
We have $a=2$ and $b=r-s.$ Therefore $p=r+s.$ The Bergman operator has the form
$$B(z,w)\lz=(1-zw^*)\lz(1-w^*z)$$
and the quasi-determinant is $\lD(z,w)=\det(1-zw^*).$
\end{example}

\begin{example} In the special case $r=1$ we obtain the {\bf unit ball} 
$$\Bl=\Bl_d=\{z\in\Cl^d:\ \<z|z\><1\}$$
of genus $p=d+1,$ where $\<z|w\>=\S_{i=1}^d z_i\o w_i$ is the inner product. The Bergman operator is given by
$$B(z,w)\lz=(1-\<z|w\>)(\lz-\<\lz|w\>z),$$
and $\lD(z,w)=1-\<z|w\>.$
\end{example}

\begin{example}\label{1.b} Let $Z=\Cl^{d+1},\ d\ge 2$ be equipped with the conjugation $z\mapsto\o z.$ Put 
$(z|w):=\S_j z_j\o w_j$ and $\<z|w\>:=2(z|w).$ The Jordan triple product
\be{1.10}\{u;v;w\}=\<u|v\>w+\<w|v\>u-\<u|\o w\>\o v\ee
makes $Z$ into a hermitian Jordan triple of rank 2 called a {\bf spin factor}. By \cite[4.16]{L1} the corresponding symmetric domain is the {\bf Lie ball}
\be{1.11}D=\{z\in Z; (z|z)<1,\ 1-2(z|z)+|(z|\o z)|^2>0\}.\ee
In the reducible case $d=1$ the Lie ball is the bidisk.
\end{example}

\begin{lemma}\label{1.a} The spin factor has the quasi-determinant
\be{1.12}\lD(z,w)=1-\<z|w\>+\f14\<z|\o z\>\<\o w|w\>=1-\<z|w\>+N(z)\o{N(w)},\ee
where 
\be{1.13}N(z)=\f12\<z|\o z\>\ee
is the Jordan algebra determinant for some maximal tripotent $e\in Z,$ normalized by $N(e)=1.$ 
\end{lemma}
\begin{proof} Write $\lD(z,w)=1-\la\<z|w\>+\lb\<z|\o z\>\<\o w|w\>$ for coefficients $\la,\lb$ to be determined. Any minimal tripotent 
$c$ has supremum norm 1 and hence belongs to $\dl D.$ Therefore
$$0=\lD(c,c)=1-\la\<c|c\>+\lb\<c|\o c\>\<\o c|c\>=1-\la$$
since $\<c|c\>=1$ and $\<c|\o c\>=0.$ It follows that $\la=1.$ On the other hand, any maximal tripotent $e$ also belongs to $\dl D,$ being the sum of two orthogonal minimal tripotents. Moreover, $\o e=\lt e$ for some constant $\lt\in\Tl:=\{u\in\Cl; |u|=1\}$; this fact is a consequence that $e=(1,0,\cdots,0)\in\Cl^{d+1}$ is a maximal tripotent and that the group $K=O(d+1)\xx\Tl$ acts transitively on the maximal tripotents. Therefore
$$0=\lD(e,e)=1-\<e|e\>+\lb\<e|\lt e\>\<\lt e|e\>=-1+4\lb$$
since $\<e|e\>=2.$ This shows that $\lb=\f14.$
\end{proof}

From now on let $\Bl=\Bl_d$ denote the unit ball in $\Cl^d.$
\begin{definition} A holomorphic map $F:\Bl\to D$ which is an isometry with respect to the Bergman metric on $\Bl$ and $D,$ resp., 
$F^\ast ds^2_D=ds^2_\Bl$, is called an {\bf isometric embedding}. Since the Bergman metric for $D$ is expressed via the Bergman operator, this is equivalent to the condition
\be{1.1}\<B_{F(z),F(w)}^{-1}F'(z)x|F'(w)y\>=\<B_{z,w}^{-1}x|y\>\ee
for all $x,y\in\Cl^d.$ 
\end{definition}

By composing with automorphisms of $\Bl$ or $D,$ which preserve the Bergman metric, we may assume that $F(0)=0.$ Since the Bergman
metrics at the origin agree, $\Bl$ and $D$ must have the same genus $p=d+1.$ The Bergman metric $B^{-1}$ on the domain $\Bl$ or $D$ is the Hessian  
$$B=\o\dl\dl\log K(z,z).$$
for the Bergman kernel $K(z,z).$ It is proved in \cite{Mok-JEMS} that the above equation \er{1.1} for the Hessians is equivalent to the equation
\be{1.2}K^D(F(z),F(z))=K^\Bl(z,z),\quad z\in\Bl,\ee
for the kernel functions. This in turn is equivalent to
\be{1.3}\lD(F(z),F(z))=1-\<z|z\>,\quad z\in\Bl.\ee
In other words, under the condition that $F(0)=0$ the isometric condition is equivalent to \er{1.3}. A stronger statement is proved in \cite{Mok-JEMS}: If $F$ is a holomorphic mapping from a neighborhood of $0\in\Bl$ to $D$ such that \er{1.3} holds, then $F$ extends to an isometry from $\Bl$ to $D$. In Section 5 we give an elementary proof of this fact.

Let $\PL(Z)$ denote the algebra of all (holomorphic) polynomials on $Z,$ endowed with the Fischer-Fock inner product. Under the natural $K$-action $\PL(Z)$ has a multiplicity-free decomposition into irreducible $K$-modules $\PL_\m(Z)\in\PL(Z)$ indexed by partitions $\m$ of length $r$ \cite{FK}. It follows that 
\be{1.4}e^{\<z|w\>}=\S_{\m}E^\m(z,w)\ee 
where 
\be{1.5}E^\m(z,w)=E^\m_w(z)
\ee
is the reproducing kernel of $\PL_\m(Z).$ The {\bf Faraut-Koranyi formula} \cite{FK}
\be{1.6}\lD(z,w)^{-\ln}=\S_{\m}(\ln)_\m E^\m(z,w)\ee
expresses powers of the quasi-determinant $\lD$ of $Z$ in terms of these kernel functions. Here $(\ln)_\m$ is the multi-variable Pochhammer symbol. Specializing \er{1.6} to $\ln=-1,$ the non-zero terms correspond exactly to partitions $\k=(1,\ldots,1,0,\ldots,0),$ with $k$ ones, for $0\le k\le r.$ It follows that
\be{1.7}\lD(z,w)=\S_{k=0}^r(-1)^k C_k E^\k(z,w),
\qquad C_k:=\P_{j=1}^k(1+\f a2(j-1)).\ee

\section{Rational isometric embeddings: Existence}
In this section we construct rational isometric embeddings in an explicit way, using the Jordan triple approach. A pair $(z,w)\in Z\xx Z$ is called {\bf quasi-invertible} if $B_{z,w}$ is invertible. In this case we have $B_{z,w}\in K^\Cl$ and
$$z^w:=B_{z,w}^{-1}(z-Q_zw)$$
is called the {\bf quasi-inverse}. The quasi-inverse map is sesqui-holomorphic in $(z,w).$ For fixed $w\in Z$ the map $z\mapsto z^w$ is a biholomorphic automorphism of the compact dual $\h Z$ of $D,$ with inverse given by $z\mapsto z^{-w}.$
Any tripotent $c\in Z$ induces a {\bf Peirce decomposition}
$$Z=Z^2\op Z^1\op Z^0,$$
where $Z^\la:=\{z\in Z:\ \{c;c;z\}=\la z\}.$ For a tripotent $c$ of rank $k,$ the Peirce 2-space $Z^2$ is itself an irreducible hermitian Jordan triple of rank $k.$ Let $D^2\ic Z^2$ denote the corresponding bounded symmetric domain.

\begin{proposition} Let $z=u+v\in Z^2\op Z^1,$ with $u\in D^2.$ Then the pair $(z,c)$ is quasi-invertible, and
$$B(z^c,z^c)=B_{z,c}^{-1}\,B_{c,v}\,B_{a,c}\,B_{v,c}B_{c,z}^{-1},$$
where $a:=u+Q_cu+\{v;v;c\}\in Z^2$ is the "real part" of $z$ in the sense of Siegel domains.
\end{proposition}
\begin{proof} We have $c^v=c$ and
$$B_{v,c}c=c-\{v;c;c\}+Q_v\,Q_cc=c-v+Q_vc$$
since $Q_cc=c.$ Using the addition formulae in \cite[Appendix, A3]{L1} this implies
\be{2.2}\begin{split}z^c&=(v+u)^c=v^c+B_{v,c}^{-1}u^{(c^v)}=v^c+B_{v,c}^{-1}u^c\\
&=B_{v,c}^{-1}(v-Q_vc)+B_{v,c}^{-1}u^c=-c+B_{v,c}^{-1}(c+u^c)\end{split}.\ee
The Peirce multiplication rules \cite[Theorem 3.13]{L1} imply $Q_vu\in Z^0$ and $Q_vv\in Z^1.$ Since $Q_c(Z^1\op Z^0)=0$ it follows that
$Q_c\,Q_vz=Q_c\,Q_v(u+v)=0.$ Since $\{c;v;u\}=0$ it follows from \cite[JP 35]{L1} that
$$B_{c,v}^{-1}z=B_{c^v,-v}z=B_{c,-v}z=u+v+\{c;v;u+v\}+Q_c\,Q_vz=u+\{c;v;v\}+v.$$
Using the addition formulae again we obtain
$$B(z^c,z^c)\,B_{c,z}=B(c+z^c,z)=B(B_{v,c}^{-1}(c+u^c),B_{c,v}(u+\{c;v;v\}+v))$$
$$=B_{v,c}^{-1}\,B(c+u^c,u+\{c;v;v\}+v)B_{v,c}$$
by applying \er{1.8} to $h=B_{v,c}\in K^\Cl.$ Since $c+u^c=c^{Q_cu}$ the addition formula \cite[A3]{L1} implies
$$(c+u^c)^w=(c^{Q_cu})^w=c^{Q_cu+w}$$
for any $w\in Z,$ and \cite[JP 33]{L1} yields
$$B_{u,c}\,B(c+u^c,w)=B_{c,Q_cu}\,B(c+u^c,w)=B_{c,Q_cu}\,B(c^{Q_cu},w)=B(c,Q_cu+w).$$
Putting $w:=u+\{c;v;v\},$ the definition of $a$ gives
$$(c+u^c)^{u+\{c;v;v\}}=c^{Q_cu+u+\{c;v;v\}}=c^a$$
and
$$B_{u,c}\,B(c+u^c,u+\{c;v;v\})=B(c,Q_cu+u+\{c;v;v\})=B_{c,a}.$$
Now $B_{z,c}=B_{u+v,c}=B_{u,c^v}\,B_{v,c}=B_{u,c}\,B_{v,c}$ by \cite[JP 34]{L1}, it follows with \cite[JP 34]{L1} that
$$B_{z,c}\,B(z^c,z^c)\,B_{c,z}\,B_{v,c}^{-1}=B_{u,c}\,B(c+u^c,u+\{c;v;v\}+v)$$
$$=B_{u,c}\,B(c+u^c,u+\{c;v;v\})\,B((c+u^c)^{u+\{c;v;v\}},v)=B_{c,a}\,B_{c^a,v}$$
$$=B(c,a+v)=B(c,v+a)=B_{c,v}\,B(c^v,a)=B_{c,v}\,B_{c,a}.$$
\end{proof}

\begin{corollary}\label{b} We have $(u+v)^c\in D$ if and only if $u+Q_cu+\{c;v;v\}\in D^2=D\ui Z^2.$
\end{corollary}
\begin{proof} By \cite[Corollary 3.15, Theorem 4.1]{L1} the domain $D$ is characterized by
$$D=\{\lz\in Z:\,B(\lz,\lz)>0\mbox{ positive definite}\}.$$
Now the proposition above implies that $B((u+v)^c,(u+v)^c)>0$ if and only if $B(c,a)>0.$ Since $a=Q_ca\in Z^2$ is self-adjoint, this is equivalent to the condition $a\in D^2.$
\end{proof}

\begin{proposition}\label{c} The quasi-determinant $\lD$ satisfies
$$\lD((u+v)^c,(u+v)^c)=|\lD(u,c)|^{-2}\,\lD(a,c).$$
\end{proposition}
\begin{proof} The Peirce multiplication rules imply for $z=z_2+z_1+z_0$
$$B_{v,c}z=z-\{v;c;z_2\}-\{v;c;z_1\}+Q_vQ_cz_2$$
with $\{v;c;z_2\}\in Z^1,\,\{v;c;z_1\}\in Z^0$ and $Q_vQ_cz_2\in Z^0,$ while all other components vanish. It follows that
\be{2.3}B_{v,c}=\cc100{-D(v,c)}10{Q_vQ_c}{-D(v,c)}1\ee
is a lower triangular matrix with respect to the Peirce decomposition $Z=\ca{Z^2}{Z^1}{Z^0}.$ Therefore we have 
$$\det B_{v,c}=1.$$ 
As a consequence, the operator $B_{u+v,c}=B(u,c^v)\,B_{v,c}=B_{u,c}\,B_{v,c}$ satisfies
$$\det B_{u+v,c}=\det B_{u,c}\,\det B_{v,c}=\det B_{u,c}.$$
Since $\lD(x,y)^p=\det B(x,y),$ the assertion follows.
\end{proof}

We now specialize to tripotents $c$ of rank 1. Then $Z^2=\Cl\,c$ and $D^2=\{uc:\ u\in\Cl,|u|<1\}$ is the unit disk. Since
$\dim Z^1=a(r-1)+b$ it follows that
$$\dim Z^2\op Z^1=1+a(r-1)+b=p-1.$$
Since Peirce spaces are orthogonal, $Z^2\op Z^1$ has the inner product $\<uc+v|u'c+v'\>=u\o u'+\<v|v'\>.$ Hence the unit ball 
$\Bl\ic Z^2\op Z^1$ has the Bergman kernel
$$K(uc+v,u'c+v')=(1-u\o u'-\<v|v'\>)^{-p},$$
where $uc+v\in Z^2\op Z^1.$

\begin{theorem} Let $uc+v\in Z^2\op Z^1$ with $|u|<1.$ Then the map
\be{2.1}(u,v)\mapsto G_c(u,v):=\(\f{uc+v}{u+1}\)^c\ee
defines a holomorphic (Bergman) isometry from the unit ball $\Bl\ic Z^2\op Z^1$ into $D.$ More precisely, we have $G_c(u,v)\in D$ if and
only if $|u|^2+\<v|v\><1,$ and the quasi-determinant satisfies
$$\lD(G_c(u,v),G_c(u',v'))=1-u\o u'-\<v|v'\>.$$
Therefore
$$K^D(G_c(u,v),G_c(u',v'))=(\lD(G_c(u,v),G_c(u',v'))^{-p}=(1-u\o u'-\<v|v'\>)^{-p}=K^\Bl(uc+v,u'c+v').$$
\end{theorem}
\begin{proof} According to Corollary \er{b} we have $G_c(u,v)\in D$ if and only if
$$\f u{u+1}+\f{\o u}{\o u+1}+\<\f v{u+1}|\f v{u+1}\><1.$$
Multiplying by $|u+1|^2$ the condition becomes
$$u(\o u+1)+\o u(u+1)+\<v|v\><|u+1|^2.$$
This is equivalent to $|u|^2+\<v|v\><1.$ The second assertion follows from Proposition \ref{c}.
\end{proof}

An explicit formula for $G_c$ follows from formula \er{2.2} and the matrix form \er{2.3} of $B_{v,c}$. Indeed we have
$$B_{v,c}^{-1}=\cc100{D(v,c)}10{-Q_vQ_c+D(v,c)^2}{D(v,c)}1$$
and 
\be{2.4}G_c(uc,v)=\(\f{uc+v}{u+1}\)^c=uc+v+\f1{1+u}Q_vc,\ee
by using \er{2.3}. We note the following identities
\be{2.5}G_c'(0)=\li,\ G_c''(0)(c,c)=0,\ G_c''(0)(c,v)=0,\ G_c''(0)(v,v)=Q_vc,\quad v\in Z^1.\ee

\begin{example} Consider the special case of matrices $Z=\Cl^{r\xx s},$ including the rectangular case where $r<s.$ We assume that 
$r>1$, since $r=1$ corresponds to the unit ball. As the construction is invariant under the action of $K$ on $Z,$ we may assume that 
$c=e_1=\bb1000$ is the first matrix unit. Then
$$Z^2\op Z^1=\{\bb uv{\t v}0:\,u\in\Cl,\,v\in\Cl^{1\xx(s-1)},\,\t v\in\Cl^{(r-1)\xx 1}\}\al\Cl^{p-1},$$
where the genus $p=2+a(r-1)+b=2+2(r-1)+(s-r)=r+s.$ For matrices, the quasi-inverse is given by
$$z^w=(1-zw^*)^{-1}z=z(1-w^*z)^{-1}.$$
It follows that the embedding $G_c$ is given by
$$G_c\bb uv{\t v}0=\bb uv{\t v}{\t v(u+1)^{-1}v}.$$
Note that 
$$G_c(z)+c=\bb{u+1}v{\t v}{\t v(u+1)^{-1}v}=\bb10{\f{\t v}{u+1}}1\bb{u+1}000\bb1{\f{v}{u+1}}01$$
is a rank 1 element of $Z.$
\end{example}

\begin{definition}\label{2.a} An isometric embedding $F:\Bl\to D\ic Z$ satisfying $F(0)=0$ will be called a {\bf Mok embedding} if there exists a minimal tripotent $c\in Z$ such that $F'(0)$ is an isometry from $\Cl^d$ onto $Z^2\op Z^1,$ and 
$$F(z)=G_c(F'(0)z)$$
for all $z\in\Bl,$ where $G_c$ is given by \er{2.1} or \er{2.4}. 
\end{definition}

If $F$ is a Mok embedding with minimal tripotent $c,$ then for all $U\in U(d)$ and $k\in K=Aut(Z),$ the composite $k\oc F\oc U$ is again a Mok embedding, with derivative $F'(0)\oc U$ and minimal tripotent $kc.$ In particular, for a constant $\lt\in\Tl$ we may put
\be{2.6}F_\lt(z):=\o\lt F(\lt z)\ee
and obtain a Mok embedding with $F'_\lt(0)=F'(0)$ and second derivative $F''_\lt(0)(v)=\o\lt F''(0)(\lt v)=\lt F''(0)(v),$ viewed as a quadratic form in $v\in\Cl^d.$

\begin{example} For the Lie ball by direct computation the mapping $G_c:\Bl\to D$ is given by
$$G_c(z_1,z')=z_1c+q(z)\o c+z',\quad z=z_1 c\op z'\in\Cl c\op Z^1\ic\Cl c\op Z^1\op\Cl\o c=Z$$
where $\{c,\o c\}$ is a frame of minimal tripotents and $q(z)=\f{(z',z')}{1+z_1}$. Then $G_c$ is an isometric embedding
satisfying the equation for the reproducing kernel \er{1.2}. This class of mappings has also been constructed in \cite{XY} using explicit coordinates.

\end{example}

\begin{remark} Considering the isometric embedding $\Bl_2\to\Bl_3: (z_1,z_2) \to (0,z_1,z_2)$ and Mok's example in the case of 
$\Bl_2\to D\ic\Cl^{2\xx 2}$, we have a non-standard isometric mapping $F:(\Bl_2,\f23 ds^2_{\Bl_2}) \to (D,ds^2_D)$ defined by
$$F(z_1,z_2)=\bb 0{z_1}{z_2}{z_1z_2}.$$
In \cite{Mok-ALM, Mok-Ng} it is shown that a non-standard holomorphic isometry $F:(\Bl_1,\ll ds_{\Bl_1}^2)\to (\Bl_1^n,ds^2_{\Bl_1^n})$ from the unit disc into a polydisc must have a singularity at some boundary point $b\in\dl\Bl_1.$ It is asked in \cite[Problem 5.2.2]{Mok-ALM} whether this is also true for isometric mappings from the unit disc into a bounded symmetric domain. Note that the isometry $z\mapsto(\f z{\F 2}, \f z{\F 2})\mapsto F(\f z{\F 2},\f z{\F 2})$ is a non-standard polynomial isometry $F:(\Bl_1,3ds_{\Bl_1}^2)\mapsto(D,ds_D^2).$ This example shows that there are non-standard polynomial maps from the unit disk and also balls into symmetric domains. This has also been observed in \cite{XY}. It would be interesting to classify general polynomial isometric maps $F:(\Bl,\ll ds_\Bl^2)\mapsto (D,ds_D^2)$ for positive integers $\ll.$
\end{remark}

\section{Differential geometric characterization of Mok type mappings}
A main result of this paper is a differential geometric characterization of Mok type embeddings $F$ in terms of a vanishing condition for the second fundamental form of the image $F(\Bl)\subset D,$ cf. Theorem \ref{3.b} and Remark \ref{3.a}. As a consequence we classify all isometric embeddings into symmetric domains of arbitrary rank, including domains not of tube type, which satisfy this condition. Here the Jordan theoretic approach is crucial in dealing with the exceptional domains. 

We fix $\Bl=\Bl_d$ the unit ball in $V:=\Cl^d.$ 

\begin{lemma}\label{3.f} Let $F:\Bl\to D$ be a Bergman isometry with $F(0)=0$. Then
$$F(z)-F'(0)z\in Ran\ F'(0)^\perp$$
\end{lemma}
\begin{proof} Let $H(z)=F(z)-F'(0)z.$ Putting $w=0$ we obtain from \er{1.1}
$$\<x|y\>=\<F'(z)x|F'(0)y\>=\<F'(0)x+H'(z)x|F'(0)y\>=\<x|y\>+\<H'(z)x|F'(0)y\>.$$
Therefore $\<H'(z)x|F'(0)y\>=0$ for all $z\in\Bl,$ showing that $Ran\ H'(z)\ic Ran\ F'(0)^\perp.$ Since $H(0)=0,$ the mean value theorem implies
$$H(z)=\I_{0}^1\ dt\ H'(tz)z\in Ran\ F'(0)^\perp.$$
\end{proof}

Let $v_1,\ldots,v_n\in V$ be tangent vectors. For a holomorphic function $f:\Bl\to Z$ we use the notation 
$$\dl^0_{v_1,\ldots,v_n}f:=f^{(n)}(0)(v_1,\cdots,v_n)$$
to denote the $n$-th derivative $f^{(n)}(0)$ as a symmetric multilinear map $f^{(n)}(0):\odot^n V\to Z$. Similarly, for a sesqui-holomorphic map $f(z,w)$ we denote the mixed derivatives of order $(m,n)$ by 
$$\dl^{0,0}_{u_1,\ldots,u_m,\o v_1,\ldots,\o v_n}f=f^{(m,n)}(0,0)(u_1,\ldots,u_m,\o v_1,\ldots,\o v_n).$$
For subsets $I=\{i_1,\ldots,i_k\}\ic\{1,\ldots,m\},\ J=\{j_1,\ldots,j_\el\}\ic\{1,\ldots,n\}$ we put $u_I:=(u_{i_1},\ldots,u_{i_k})$ and $(u_I,\o v_J):=(u_{i_1},\ldots,u_{i_k},\o v_{j_1},\ldots,\o v_{j_\el}).$ Here the vectors $u_1,\ldots,u_m$ or $v_1,\ldots,v_n$ need not be distinct.

\begin{lemma}\label{3.d} Let $F,G:\Bl\to D$ be two Bergman isometries preserving the origin, which agree up to order $n\ge 0$ at $0$,
i.e. $F^{(j)}(0)=G^{(j)}(0)$, $0\le j\le n$. Then
$$\<\dl^0_{u_0,\ldots,u_m}F|\dl^0_{v_0,\ldots,v_n}F\>=\<\dl^0_{u_0,\ldots,u_m}G|\dl^0_{v_0,\ldots,v_n}G\>$$
for $m\le n$ and all vectors $u_i,v_j\in V.$ In other words, we have 
$$F^{(n+1)}(0)^*F^{(m+1)}(0)=G^{(n+1)}(0)^*G^{(m+1)}(0):\odot^{m+1}\Cl^d\to Z.$$
\end{lemma}
\begin{proof} Put 
\be{3.1}B_F(z,w)\lz:=B(F(z),F(w))\lz=\lz-\{F(z);F(w);\lz\}+Q_{F(z)}Q_{F(w)}\lz.\ee
Then we have for $I,J\ic\{1,\ldots,n\}$
$$\dl^{0,0}_{u_I,\o v_J}B_F\lz=-\{\dl^0_{u_I}F;\dl^0_{v_J}F;\lz\}
+\f14\S_{I'\ic I}\S_{J'\ic J}\{\dl^0_{u_{I'}}F; \{\dl^0_{v_{J'}}F;\lz;\dl^0_{v_{J\sm J'}}F\};\dl^0_{u_{I\sm I'}}F\}.$$
Since all derivatives involved are of order $\le n$ the assumption implies
$$\dl^{0,0}_{u_I,\o v_J}B_F=\dl^{0,0}_{u_I,\o v_J}B_G.$$
Using the formula
$$\dl^0_{u_I}B_F^{-1}=\S_{I_1\iu\ldots I_k=I}\pm B_F^{-1}(\dl^0_{u_{I_1}}B_F)B_F^{-1}\ldots B_F^{-1}(\dl^0_{u_{I_k}}B_F)B_F^{-1}
=\S_{I_1\iu\ldots I_k=I}\pm(\dl^0_{u_{I_1}}B_F)\ldots(\dl^0_{u_{I_k}}B_F)$$
at $z=w=0,$ we see that the inverse has the same property:
\be{3.2}\dl^{0,0}_{u_I,\o v_J}B_F^{-1}=\dl^{0,0}_{u_I,\o v_J}B_G^{-1}.\ee
(The sign $\pm$  depends on the degree of the differentiations.) Let $M,N\ic\{1,\ldots,n\}.$ The isometry condition \er{1.1} implies
\be{3.3}\dl^{0,0}_{u_M,\o v_N}\<B_{z,w}^{-1}u_0|v_0\>=\dl^{0,0}_{u_M,\o v_N}\<B_{F(z),F(w)}^{-1}F'(z)u_0|F'(w)v_0\>$$
$$=\S_{I\ic M}\S_{J\ic N}\<\dl^{0,0}_{u_I,\o v_J}B_F^{-1}(\dl^0_{u_0,u_{M\sm I}}F)|\dl^0_{v_0,v_{N\sm J}}F\>$$
$$=\<\dl^0_{u_0,u_M}F|\dl^0_{v_0,v_N}F\>+\S_{I\ne\OO}\S_{J\ne\OO}\<\dl^{0,0}_{u_I,\o v_J}B_F^{-1}(\dl^0_{u_0,u_{M\sm I}}F)|\dl^0_{v_0,v_{N\sm J}}F\>.\ee
In view of \er{3.2}, the second sum for $F$ and $G$ is the same, since all derivatives are of order $\le n$ at $0.$ Since $G$ also satisfies \er{1.1}, the first terms also agree.
\end{proof}

\begin{lemma} Let $F:\Bl\to D$ be a Bergman isometry with $F(0)=0,$ and let $u,v,v_1,\ldots,v_n\in V.$ Let $N=\{1,\ldots,n\}$ and 
$x,y\in V.$ Then
\be{3.4}\<\dl^{0,0}_{u,\o v_N}B_{z,w}^{-1}x|y\>=\<\dl^0_{x,u}F|\dl^0_{y,v_N}F\>
+\S_{J\ne\OO}\<\{\dl^0_uF;\dl^0_{v_J}F;\dl^0_xF\}|\dl^0_{y,v_{N\sm J}}F\>,\ee
\be{3.5}\<\{u;v;x\}|y\>=\<\dl^0_{u,x}F|\dl^0_{v,y}F\>+\<\{\dl^0_uF;\dl^0_vF;\dl^0_xF\}|\dl^0_yF\>\ee
\end{lemma}
\begin{proof} Any proper partition of the indices $\{u,v_1,\ldots,v_n\}$ involves at least one subset not containing $u$, for which the derivative of $B_F$ at $(0,0)$ vanishes. Hence
  \begin{equation}
    \label{eq:3.5-1}
  \dl^{0,0}_{u,\o v_J}B_F^{-1}=-\dl^{0,0}_{u,\o v_J}B_F.
  \end{equation}
Since $F(0)=0,$ \er{3.1} implies
$$\dl^{0,0}_{u,\o v_J}B_F\lz=-\{\dl^0_u F;\dl^0_{v_J}F;\lz\}
+\f12\S_{J'\ic J}\{\dl^0_u F\{\dl^0_{v_{J'}}F;\lz;\dl^0_{v_{J\sm J'}}F\}F(0)\}=-\{\dl^0_u F;\dl^0_{v_J}F;\lz\}.$$
For $|M|=1$ the formula \er{3.3} simplifies to
\be{3.6}\<\dl^{0,0}_{u,\o v_N}B_{z,w}^{-1}x|y\>=\<\dl^0_{x,u}F|\dl^0_{y,v_N}F\>+\S_{J\ne\OO}\<\dl^{0,0}_{u,\o v_J}B_F^{-1}(\dl^0_xF)|\dl^0_{y,v_{N\sm J}}F\>$$
$$=\<\dl^0_{x,u}F|\dl^0_{y,v_N}F\>+\S_{J\ne\OO}\<\{\dl^0_uF;\dl^0_{v_J}F;\dl^0_xF\}|\dl^0_{y,v_{N\sm J}}F\>.\ee
This proves the first assertion. For $|N|=1$ we obtain from 
\ref{eq:3.5-1}
$$\dl^{0,0}_{u,\o v}B_F^{-1}\lz=-\dl^{0,0}_{u,\o v}B_F\lz=\{\dl^0_uF;\dl^0_vF;\lz\},$$ 
and  \er{3.6} simplifies to
$$\<\dl^0_{x,u}F|\dl^0_{y,v}F\>+\<\{\dl^0_uF;\dl^0_vF;\dl^0_xF\}|\dl^0_yF\>=\<\dl^{0,0}_{u,\o v}B_{z,w}^{-1}x|y\>=-\<\dl^{0,0}_{u,\o v}B_{z,w}x|y\>=\<\{u;v;x\}|y\>,$$
proving the second assertion.
\end{proof}
\begin{remark}\label{3.a} According to \cite[Theorem 2.10]{L1} the {\bf curvature} $R_D$ of an irreducible bounded symmetric domain $D$ at the origin has the sesqui-linear part 
$$R_D(u,v)x=-\{u;v;x\}.$$ 
Thus, in terms of the Bergman curvature tensors we may express \er{3.5} as
$$\<R_\Bl(u',v')x'|y'\>=\<R_D(u,v)x|y\>+\<F''(0)(x,u)|F''(0)(y,v)\>,$$
where $u'=F'(0)u,v'=F'(0)v,x'=F'(0)x,y'=F'(0)y.$ In other words, $F''(0)(x,u)$ is the 'second fundamental form' of the submanifold 
$F(\Bl)\ic D$ at $0=F(0)\in D.$ 
\end{remark}

The main result of this section is the following differential-geometric characterization of Mok embeddings.
\begin{theorem}\label{3.b} Let $D$ be an irreducible hermitian bounded symmetric domain of arbitrary rank. (In particular, the two exceptional domains are admitted.) Let $F:\Bl\to D\ic Z$ be an isometric embedding with $F(0)=0.$ Then $F$ is a Mok embedding 
(see Definition \ref{2.a}) if and only if there exists a unit vector $e\in V$ such that the second fundamental form (cf. the remark \ref{3.a} above) satisfies 
\be{3.28}\dl^0_eF'=F''(0)(e,\cdot)=0.\ee 
\end{theorem} 

It follows easily from \er{2.5} that \er{3.28} holds for any Mok embedding. The proof of the converse direction will be given in several steps.

\begin{lemma} \label{3.e} Let $F:\Bl\to D$ be a Bergman isometry and $F(0)=0$. Let $e\in V$ be a unit vector. Then
\be{3.7}\dl^0_{e,e}F=0\mbox{ if and only if }c:=\dl^0_eF\mbox{ is a minimal tripotent}.\ee
\be{3.8}\dl^0_eF'=F''(0)(e, \cdot)=0\mbox{ if and only if, in addition, } F'(0)V=Z^2_c\op Z^1_c.\ee
\end{lemma}
\begin{proof} Write $c:=\dl^0_eF=\S_i\ll_i e_i$ for minimal orthogonal tripotents $e_i\in Z,$ with $\ll_i>0.$ Since $F'(0)$ is an isometry, we have $\S_i\ll_i^2=1.$ As a special case of \er{3.5} we have
$$\f12\|\dl^0_{e,e}F\|^2=\<Q_ee|e\>-\<Q_{\dl^0_eF}(\dl^0_eF)|\dl^0_eF\>=1-\<Q_cc|c\>=1-\S_i\ll_i^4.$$
Thus $\dl^0_{e,e}F=0$ if and only if $\S_i\ll_i^4=1.$ Hence $\S_i\ll_i^2(1-\ll_i^2)=0,$ showing that all $\ll_i=1,0.$
Since $\S_i\ll_i^2=1,$ only one term $\ll_i=1.$ Thus $c$ is a minimal tripotent. This proves the first assertion. By \er{3.5} we have
\be{3.9}\|\dl^0_{e,v}F\|^2=\<\{e;e;v\}|v\>-\<\{c;c;\dl^0_vF\}|\dl^0_vF\>\ee
for all $v\in V.$ Thus $\dl^0_{e,v}F=0$ if and only if
\be{3.10}\<\{c;c;\dl^0_vF\}|\dl^0_vF\>=\<\{e;e;v\}|v\>.\ee
For $v=e$ this was shown above. Now let $v\in e^\perp.$ Then $\<\dl^0_vF|c\>=\<\dl^0_vF|\dl^0_eF\>=\<v|e\>=0$ since $F'(0)$ is isometric. Thus $\dl^0_vF\in c^\perp=Z^1\op Z^0.$ Write $\dl^0_vF=w_1+w_0.$ Then
$$\<w_1|w_1\>+\<w_0|w_0\>=\|\dl^0_vF\|^2=\|v\|^2=\<\{e;e;v\}|v\>.$$
Since $\<\{c;c;\dl^0_vF\}|\dl^0_vF\>=\<\{c;c;w_1\}|w_1\>=\<w_1|w_1\>,$ the condition \er{3.10} is equivalent to $w_0=0$, namely
$\dl_v^0F\in Z^1$, and finally $F'(0)V=Z_c^2\op Z_c^1$.
\end{proof}

For the rest of this section we suppose $F:\Bl\to D$ is an isometric embedding satisfying \er{3.8} and $d=\dim(Z_c^2\op Z_c^1)$ for some (and hence any) minimal tripotent $c.$ Then $F'(0)$ is an isometry from $V$ into $Z^2_c\op Z^1_c,$ which is surjective since the dimensions agree. Hence $F$ and the Mok embedding $G_c$ (cf \er{2.4}) agree up to order 1 at $0.$ Lemma \ref{3.d} implies
\be{3.11}\<\dl^0_{u,x}F|\dl^0_{v,y}F\>=\<\dl^0_{u,x}G_c|\dl^0_{v,y}G_c\>\ee
for all $u,x,v,y\in V.$ Thus we may identify $V\al Z^2\op Z^1,$ so that $\Bl$ becomes the unit ball of $Z^2\op Z^1,$ and 
$F'(0)$ is the inclusion map. With Lemma \ref{3.f} we have
\be{3.13}F(z)=z+H(z)\ee
for all $z\in\Bl\ic Z^2\op Z^1,$ with $H(z)\in Z^0.$ This implies $\dl^0_{x,y}F\in Z^0$ for $x,y\in Z^1.$ 

\begin{lemma}\label{3.g} $Z^0$ is spanned by $\dl^0_{x,y}F$ for $x,y\in Z^1.$
\end{lemma}
\begin{proof} Regarding $F''(0)$ as a linear map from $Z^1\odot Z^1$ into $Z^0,$ \er{3.11} can be expressed as
$$F''(0)^*F''(0)=G''_c(0)^*G''_c(0).$$
Since $G_c''(0):Z^1\odot Z^1\to Z^0$ is surjective, it follows that
$$\dim Z^0=\mbox{rank}\ G_c''(0)^*G_c''(0)=\mbox{rank}\ F''(0)^*F''(0).$$
Therefore $F''(0):Z^1\odot Z^1\to Z^0$ is also surjective. 
\end{proof}

\begin{lemma}\label{3.h} Restricted to $B\ui Z^1$ we have $F(v)=v+\dl^0_{v,v}F$ for all $v\in Z^1.$ 
\end{lemma}
\begin{proof} Let $n\ge 2$ and $v,v_1,\ldots,v_n\in Z^1.$ By \er{3.4} we have
$$\<\dl^0_{x,y}F|\dl^0_{v,v_N}F\>=\S_{J\ne\OO}\<\{\dl^0_y F;\dl^0_{v_J}F;\dl^0_x F\}|\dl^0_{v,v_{N\sm J}}F\>
=\S_{J\ne\OO}\<\{y;\dl^0_{v_J}F;x\}|\dl^0_{v,v_{N\sm J}}F\>$$
for $x,y\in Z^1$ and $N=\{1,\ldots,n\}.$ If $|J|\ge 2,$ then $\dl^0_{v_J}F\in Z^0$ and hence $\{y;\dl^0_{v_J}F;x\}\in Z^2.$ If $J=N,$ then $\dl^0_{v,v_{N\sm J}}F=\dl^0_v F=v\in Z^1.$ If $J\ne N,$ then $\dl^0_{v,v_{N\sm J}}F\in Z^0.$ If $|J|=1,$ then 
$\{\dl^0_y F;\dl^0_{v_J}F;\dl^0_x F\}=\{y;v_J;x\}\in Z^1$ and $\dl^0_{v,v_{N\sm J}}F\in Z^0.$ In all cases, we have orthogonality
\be{3.12}\<\dl^0_{x,y}F)|\dl^0_{v,v_N}F\>=0.\ee
Since $\dl^0_{x,y}F$ spans all of $Z^0$ by Lemma \ref{3.g}, it follows that $\dl^0_{v,v_N}F=0.$ Thus $F^{(n+1)}(0)(v,\cdots,v)=0$ for all
$v\in Z^1,$ proving our claim.
\end{proof}

A 2-homogeneous polynomial $p$ on Z has a polar form
\be{3.14}\t p(x,y):=p(x+y)-p(x)-p(y)\ee
which is bilinear and symmetric in $x,y\in Z.$

\begin{lemma}\label{3.i} Let $v\in Z^1.$ Then
\be{3.15}\<v|v\>\<\dl^0_{v,v}F|\dl^0_{v,v}F\>=\<\{v;\dl^0_{v,v}F;\dl^0_{v,v}F\}|v\>\ee
and
\be{3.16}\t p(v,\dl^0_{v,v}F)=0.\ee
for any polynomial $p\in\PL_{1,1}(Z).$
\end{lemma}
\begin{proof} For $0\le k\le r$ the Fischer-Fock kernel for the partition $\k:=(1,\ldots,1,0,\ldots,0)$ can be written as
$$E^\k(x,y)=\S_\la p_{k,\la}(x)\ \o{p_{k,\la}(y)}$$
where $p_{k,\la}$ is an orthonormal basis of the associated Fischer-Fock space $\PL_\k(Z).$ Write
$$p_{k,\la}(v+w)=\S_{i=0}^k p_{k,\la}^i(v,w)$$
for homogeneous polynomials $p_{k,\la}^i(v,w)$ of bidegree $(k-i,i).$ Then $\t p_{2,\la}(v,w)=p_{2,\la}^1(v,w).$ By 
\cite[p. 3030,(5)]{L2} we have 
$$\f1{2+a}\(\<x|y\>^2-\<Q_xy|y\>\)=E^{1,1}(x,y)=\S_\la p_{2,\la}(x)\ \o{p_{2,\la}(y)},$$
and  the term $C_2E^{1,1}$ in (\ref{1.7}) is
\begin{equation*}
C_2E^{1,1}(x,y)
=\frac 12\(\<x|y\>^2-\<Q_xy|y\>\)
\end{equation*}
since $C_2=1+\frac a2=\frac{2+a}2$. For $v\in Z^1,w\in Z^0$ we obtain 
$Q_{v+w}(v+w)=Q_vv+\{v;w;w\}+\{v;v;w\}+Q_ww+Q_vw$ since $Q_wv=0$ by the Peirce rules.
 Hence $\<Q_{v+w}(v+w)|v+w\>=\<Q_v|v\>+\<\{v;w;w\}|v\>+\<\{v;v;w\}|w\>+\<Q_ww|w\>=\<Q_vv|v\>+2\<\{v;w;w\}|v\>+\<Q_ww|w\>,$ since $Q_vw\in Z^2$ and $\<\{v;v;w\}|w\>=\<v|\{v;w;w\}\>
=\<v|\{w;w;v\}\>=\<\{v;w;w\}|v\>.$ It follows that
\begin{equation}
\label{c-2-e-2}
\begin{split}
&\qquad 2C_2E^{1,1}(v+w, v+w)\\
&=\<v+w|v+w\>^2-\<Q_{v+w}(v+w)|v+w\>\\
&=\<v+w|v+w\>^2-\<Q_vv|v\>-2\<\{v;w;w\}|v\>-\<Q_ww|w\>\\
&=\(\<v|v\>+\<w|w\>\)^2-\<Q_vv|v\>-2\<\{v;w;w\}|v\>-\<Q_ww|w\>\\
&=\<v|v\>^2+\<w|w\>^2+2\<v|v\>\<w|w\>-\<Q_vv|v\>-2\<\{v;w;w\}|v\>-\<Q_ww|w\>.
\end{split}
\end{equation}
Comparing terms of bidegree $(1,1)$ in $v$ and in $w$ yield
\be{3.18}
2(\<v|v\>\<w|w\>-\<\{v;w;w\}|v\>)=(a+2)\S_\la\t p_{2,\la}(v,w)\ \o{\t p_{2,\la}(v,w)}.\ee
Now let $k\ge 3.$ Since $v\in Z^1$ has rank $\le 2$ we have $p_{k,\la}(v)=0$ which implies that the term $p_{k,\la}^0(v,w)$ independent of $w$ vanishes. Applying Lemma \ref{3.h} to $v\in B\ui Z^1$ and putting $w:=\dl^0_{v,v}F\in Z^0,$ we obtain from \er{3.13} and \er{1.7}  
that
\begin{equation*}
\begin{split}
1-\<v|v\>&=\lD(F(v),F(v))=1-\<v+w|v+w\>+\S_{k=2}^r(-1)^k\ C_k\ E^\k(v+w,v+w)\\
&=1-\<v|v\>-\<w|w\>+\S_{k=2}^r(-1)^k\ C_k\ E^\k(v+w,v+w).
\end{split}
\end{equation*}
Hence, solving for $\<w|w\>$ and using \er{c-2-e-2}, we obtain
$$\<w|w\>=\S_{k=2}^r(-1)^k\ C_k\ E^\k(v+w,v+w)=\f12\(\<v|v\>^2+\<w|w\>^2-\<Q_vv|v\>-\<Q_ww|w\>\)$$
$$+\<v|v\>\<w|w\>-\<\{v;w;w\}|v\>+\S_{k=3}^r(-1)^k\ C_k\S_\la\S_{i,j}p_{k,\la}^i(v,w)\ \o{p_{k,\la}^j(v,w)}.$$
Comparing terms of homogeneity $(3,\o 3)$ in $v$ implies $\<v|v\>\<w|w\>-\<\{v;w;w\}|v\>=0,$ since for $k\ge 3$ $p_{k,\la}^i(v,w)$ has degree $(k-i)+2i>3$ as a polynomial in $v.$ This proves \er{3.15}. Using \er{3.18} we obtain $\t p_{2,\la}(v,w)=0$ for all 
$\la,$ and \er{3.16} follows by linearity. 
\end{proof} 

The Mok type embedding $G=G_c$ has the second derivative
$$\dl^0_{v,v}G=Q_v c=\f12\{v;c;v\}$$
for all $v\in Z^1.$ The crucial step in proving Theorem \ref{3.b} is the following
\begin{lemma}\label{3.j} Let $Z$ be a hermitian Jordan triple. Then there exists a constant $\lt\in\Tl$ such that 
$$\dl^0_{v,v}F=\lt\ Q_vc$$
for all $v\in Z^1.$
\end{lemma}
\begin{proof}Consider $w=\dl^0_{v,v}F$ in the proof above
as a quadratic mapping from $Z^1$ to $Z^0.$ By \er{3.11} we have unitarity
\be{3.19}\<w|w\>=\<Q_vc|Q_vc\>.\ee
Consider first the {\bf spin factors} $Z.$ Since $\dim Z^0=1,$ there exists a rational function $r(v)$ on $Z^1$ such that $w=r(v)\ Q_cv$ 
for all $v\in Z^1.$ Then \er{3.19} implies $|r(v)|=1$ for all $v.$ Thus $r(v)$ is constant, as asserted.\\ 
For {\bf symmetric matrices} $Z=\Cl^{r\xx r}_{sym}$ write $v+w=\bb0{v_\prime^T}{v_\prime}w,$ with $v_\prime=(v_2,\ldots,v_r)^T$ and 
$w=(w_{ij})\in\Cl^{(r-1)\xx(r-1)}_{sym}.$ Let $e_i$ denote the $i$-th unit row vector, and let $w_i$ be the $i$-th row of $w.$ We have 
$$\{v;v;w\}=\bb000{v_\prime v_\prime^*w+w\o v_\prime v_\prime^T}$$
and
$$v_\prime^*v_\prime\ I_{r-1}-v_\prime v_\prime^*=\S_{2\le i<j\le n}(v_j\ e_i-v_i\ e_j)^*(v_j\ e_i-v_i\ e_j).$$
By \er{3.15} it follows that
$$0=\<v|v\>\<w|w\>-\<\{v;w;w\}|v\>=2\<v_\prime|v_\prime\>\ \<w|w\>-tr(v_\prime v_\prime^*w+w\o v_\prime v_\prime^T)w^*
=2\ tr\ w^*(v_\prime^*v_\prime\ I_{r-1}-v_\prime v_\prime^*)w.$$
Hence the positive matrix
$$w^*(v_\prime^*v_\prime\ I_{r-1}-v_\prime v_\prime^*)w=\S_{2\le i<j\le n}w^*(v_j\ e_i-v_i\ e_j)^*(v_j\ e_i-v_i\ e_j)w$$
$$=\S_{2\le i<j\le n}(v_j\ w_i-v_i\ w_j)^*(v_j\ w_i-v_i\ w_j)$$
vanishes. It follows that $v_j\ w_i=v_i\ w_j$ for all $2\le i<j\le r.$ Therefore 
\be{3.20}v_j\ w_{ik}=v_i\ w_{jk}\ee
for all $2\le i<j\le r$ and $2\le k\le r.$ There exists a rational function $r(v)$ on $Z^1$ such that $w_{22}=r(v)\ v_2^2$ for all 
$v\in Z^1.$ With \er{3.20} we conclude that $w=r(v)\ Q_vc,$ and \er{3.19} implies that $r(v)$ is constant.\\
The preceding argument can be adopted to the rectangular matrix case and, with some effort, to the anti-symmetric matrix case. We prefer instead to use the {\bf grid-theoretic} approach to Jordan triples \cite{N} which also works for the exceptional cases. A {\bf grid} $\EL$ in an irreducible hermitian Jordan triple $Z$ is a basis consisting of tripotents with certain combinatorial properties \cite[Chapter I.4]{N}. If $\EL$ is a grid and $c\in\EL,$ then $\EL$ is compatible with the Peirce decomposition relative to $c.$ The tripotents $f\in\EL^\la:=\EL\ui Z^\la$ constitute a basis of $Z^\la.$ The characteristic multiplicity $a$ of $Z$ determines the kind of underlying grid. For symmetric matrices $Z=\Cl^{r\xx r}_{sym}$ we have $a=1,$ and spin factors of dimension $d$ have $a=d-2.$ Thus $a$ is odd in the odd-dimensional case. For all remaining cases $a$ is even and $Z$ is spanned by a so-called 'ortho-collinear' grid $\EL.$ Assume now that $a$ is even. By definition \cite[p.~12, p.16]{N},
two tripotents $e, f$ are called {\bf collinear} if
$e\in Z^1_f, f\in Z^1_e$, and an ordered quadruple $(e_1,e_2,e_3,e_4)$ of tripotents is called a {\bf quadrangle} if $e_i,e_{i+1}$ are collinear, $e_i,e_{i+2}$ are orthogonal and $\{e_i;e_{i+1};e_{i+2}\}=e_{i+3}$ for all indices $i$ modulo $4.$\\ 
Fix a tripotent $f\in\EL^0$ and define a rational function $r(v)$ 
(depending
on $f$ a priori) on $Z^1$ by $\<w|f\>=r(v)\ \<Q_vc|f\>.$ We will show that 
\be{3.21}
\<w-r(v)\ Q_vc|f'\>=0\ee
for all $f'\in\EL^0$ which are collinear to $f.$ Since any pair $f_1,f_2\in\EL^0$ can be connected by a chain of collinear tripotents in $\EL^0$ \er{3.21} will imply that $\<w|f\>=r(v)\ \<Q_vc|f\>$ for all $f\in\EL^0.$ Hence $w=r(v)\ Q_vc$ and \er{3.19} shows 
$|r(v)|=1$ for all $v.$ Hence $r(v)=\ll\in\Tl$ is constant, concluding the proof of Lemma \ref{3.j}.\\ 
In order to prove \er{3.21} choose $e\in\EL^1$ such that $e\bot f$ and the Peirce 2-space of $e+f$ contains $f'.$ According to \cite[Theorem 2.13]{N} $Z_{e+f}^2$ is a spin factor of (even) dimension $a+2,$ which is spanned by $a/2$ quadrangles $(e,e^j,f,f^j),\ 1\le j\le a/2,$ with $e^j\in\EL^1$ and $f^j\in\EL^0.$ We call this a {\bf presentation}
\be{3.22}\bb e{e^1/\ldots/e^{a/2}}{f^1/\ldots/f^{a/2}}f\ee
of the Peirce 2-space. By \cite[p. 75, (2.8)]{N}, an even-dimensional spin factor has the grid $\{e_j^\pm:\ 1\le j\le m/2\},$ with quadrangles $(e_1^+,e_j^+,e_1^-,-e_j^-)$ for $j>1.$ This gives a presentation
$$\bb{e_1^+}{e_2^+/\ldots/e_{m/2}^+}{-e_2^-/\ldots/-e_{m/2}^-}{e_1^-}.$$
The spin factor determinant, normalized at $e_1^++e_1^-,$ has the form
$$N(z)=\S_{j=1}^{m/2}\<z|e_j^+\>\<z|e_j^-\>=\<z|e_1^+\>\<z|e_1^-\>-\S_{j=2}^{m/2}\<z|e_j^+\>\<z|-e_j^-\>.$$
Applied to \er{3.22} it follows that the spin factor determinant $N$ of $Z_{e+f}^2,$ normalized by $N(e+f)=1,$ is given by
$$N(z)=\<z|e\>\<z|f\>-\S_j\<z|e^j\>\<z|f^j\>.$$
Let $P:Z\to Z_{e+f}^2$ be the Peirce 2-projection. Then $N_P:=N\oc P$ belongs to $\PL_{1,1}(Z)$ and thus, by \er{3.16} has a vanishing polarization $\t N_P(v,w)=0.$ Since $P$ is compatible with the Peirce decomposition relative to $c,$ this implies
\be{3.23}\<v|e\>\<w|f\>=\S_j\<v|e^j\>\<w|f^j\>.\ee
Similarly, $\t N_P(v,Q_vc)=0$ showing that
\be{3.24}\<v|e\>\<Q_vc|f\>=\S_j\<v|e^j\>\<Q_vc|f^j\>.\ee
Combining \er{3.23} and \er{3.24} we obtain
\be{3.25}\S_j\<v|e^j\>\<\t w|f^j\>=0\ee
where, for the rest of the proof, we put $\t w:=w-r(v)\ Q_vc.$\\ 
To continue, we consider the separate cases. The {\bf rectangular matrices} $Z=\Cl^{r\xx s}$ have the matrix unit grid $\EL=\{E_{ij}:\ 1\le i\le r,\ 1\le j\le s\}.$ Put $c:=E_{11}.$ Then $\EL^0=\{E_{ij}:\ 2\le i\le r,\ 2\le j\le s\}$ and $\EL^1=\{E_{i1},E_{1j}:\ 2\le i\le r,\ 2\le j\le s\}.$ Consider the presentation
$$\bb{E_{1j}}{E_{1\el}}{E_{kj}}{E_{k\el}},\ \bb{E_{i1}}{E_{k1}}{E_{i\el}}{E_{k\el}}$$
of the Peirce 2-spaces of $E_{1j}+E_{k\el}$ and $E_{i1}+E_{k\el}$ involving $f:=E_{k\el}\in\EL^0.$ Define $z_{ij}:=\<z|E_{ij}\>.$ Then  \er{3.25} yields $v_{1\el}\t w_{kj}=v_{k1}\t w_{i\el}=0.$ Since $v_{1\el}v_{k1}\ne 0$ on a dense open subset of $v\in Z^1$ the assertion $\t w_{kj}=\t w_{i\el}=0$ follows, 
proving (\ref{3.21}).
\\
The {\bf anti-symmetric matrices} $Z=\Cl^{n\xx n}_{asym}$ have the symplectic grid $F_{ij}=E_{ij}-E_{ji}=-F_{ji},$ where $1\le i<j\le n.$ Put $c:=F_{12}.$ Then $\EL^0=\{F_{ij}:\ 3\le i<j\le n\}$ and $\EL^1=\{F_{ai}:\ a=1,2,\ i=3,\ldots n\}.$ For $n=4$ we obtain a spin factor. Now suppose $n\ge 5.$ By \cite[Chapter II, (2.4)]{N} the quadrangles have the form $(F_{ij},F_{kj},F_{k\el},F_{i\el}).$ For 
$a=1,2$ we obtain a presentation
$$\bb{F_{ia}}{F_{ja}/F_{ka}}{F_{ik}/F_{ji}}{F_{jk}}$$
of the 6-dimensional Peirce 2-space of $F_{ia}+F_{jk}$ involving $f:=F_{jk}\in\EL^0.$ Define $z_{ij}:=\<z|F_{ij}\>.$ Then \er{3.25} yields two equations written in matrix form
$$\bb{v_{j1}}{v_{k1}}{v_{j2}}{v_{k2}}\ba{\t w_{ik}}{\t w_{ji}}=0.$$
Using linear independence on the dense open subset $v_{j1}v_{k2}\ne v_{j2}v_{k1}$ of $v\in Z^1$ we obtain the assertion $\t w_{ik}=\t w_{ji}=0.$\\
The {\bf exceptional Jordan triple} $Z=\Ol_\Cl^{1\xx 2}$ of rank 2 has the bi-Cayley grid $\EL=\{e_i^\pm:\ 1\le i\le 8\}.$ Put 
$c=e_1^+.$ Then $\EL^0=\{e_1^-,e_5^-,e_{k+4}^+:\ k=2,3,4\}$ and $\EL^1=\{e_5^+,e_k^\pm,e_{k+4}^-:\ k=2,3,4\}.$ The quadrangle relations \cite[Chapter II, (3.3),(3.5)]{N} yield presentations
$$\bb{-e_5^+}{e_2^-/e_3^-/e_4^-}{e_6^+/e_7^+/e_8^+}{e_1^-},\quad
\bb{e_{k+4}^-}{e_k^-/e_{k'}^+/e_{k''}^+}{e_5^-/-e_{k''+4}^+/e_{k'+4}^+}{e_1^-}$$
of the four 8-dimensional Peirce 2-spaces involving $f:=e_1^-.$ Here $\{k,k',k''\}=\{2,3,4\}$ in cyclic order. Define $z_i^\Le:=\<z|e_i^\Le\>.$ From \er{3.25} we obtain four equations written in matrix form
$$\begin{pmatrix}0&v_2^-&v_3^-&v_4^-\\v_2^-&0&v_4^+&-v_3^+\\v_3^-&-v_4^+&0&v_2^+\\v_4^-&v_3^+&-v_2^+&0&\end{pmatrix}
\da{\t w_5^-}{\t w_6^+}{\t w_7^+}{\t w_8^+}=0.$$
Since the matrix has non-zero determinant on the dense open subset $v_2^-v_2^++v_3^-v_3^++v_4^-v_4^+\ne 0,$ we obtain the assertion
$\t w_5^-=\t w_6^+=\t w_7^+=\t w_8^+=0.$\\ 
The {\bf exceptional Jordan algebra} $Z=\HL_3^\Cl(\Ol)$ has the Albert grid $\EL=\{[1],[2],[3]\}\iu\{[ij]_r^\pm:\ 1\le i<j\le 3,\ 1\le r\le 4\}.$ Put $c=[1].$ Then $\EL^1=\{[1a]_r^\pm:\ a=2,3,\ 1\le r\le 4\}$ and $\EL^0=\{[2],[3]\}\iu\{[23]_r^\pm:\ 1\le r\le 4\}.$ The quadrangle relations \cite[Chapter II, (3.13),(3.14),(3.15)]{N} for $\Le=\pm,$ yield presentations
$$\bb{[13]_1^\Le}{[12]_1^\Le/[12]_r^{-\Le}}{[23]_1^\Le/-[23]_r^\Le}{[2]},$$
$$\bb{[13]_r^\Le}{[12]_1^{-\Le}/[12]_r^\Le/[12]_{r'}^{-\Le}/[12]_{r''}^{-\Le}}{[23]_r^\Le/[23]_1^\Le/[23]_{r''}^{-\Le}/-[23]_{r'}^{-\Le}}{[2]}$$
of the eight 10-dimensional Peirce 2-spaces involving $f:=[2].$ Here $\{r,r',r''\}=\{2,3,4\}$ in cyclic order. Define 
$v_\el^\Le:=\<v|[12]_\el^\Le\>$ and $\t w_\el^\Le:=\<\t w|[23]_\el^\Le\>.$ Then \er{3.25} yields eight equations 
\be{3.26}\t w_1^\Le v_1^\Le -\S_{r\ge 2}\t w_r^\Le v_r^{-\Le}=0,\ee
\be{3.27}w_1^\Le v_r^\Le+\t w_r^\Le v_1^{-\Le}+\t w_{r''}^{-\Le}v_{r'}^{-\Le}-\t w_{r'}^{-\Le}v_{r''}^{-\Le}=0,\ r\ge 2.\ee
Now consider the split octonions $\Ol_\Cl,$ with basis $\{c_\el^\pm:\ 1\le\el\le 4\}$ as defined in \cite[Chapter III.3.1]{N}.
By the multiplication table \cite[III.(3.6)]{N} the product of elements $x=\S_{\el=1}^4\S_{\Le=\pm}x_\el^\Le c_\el^\Le\in\Ol_\Cl$ has components
$$(x\cdot y)_1^\Le=x_1^\Le y_1^\Le -\S_{r\ge 2}x_r^{-\Le}y_r^\Le,$$
$$(x\cdot y)_r^\Le=x_r^\Le y_1^\Le+x_1^{-\Le}y_r^\Le+x_{r''}^{-\Le}y_{r'}^{-\Le}-x_{r'}^{-\Le}y_{r''}^{-\Le},\ r\ge 2.$$
Therefore the equations \er{3.26} and \er{3.27} are equivalent to $\u{\t w}\cdot\u v=0,$ where $\u x\in\Ol_\Cl$ is defined by 
$\u x_1^\Le:=x_1^\Le,\ \u x_r^\Le:=x_r^{-\Le},\ r\ge 2.$ Since $\u v$ is not a zero-divisor in $\Ol_\Cl$ for a dense open subset of 
$v\in Z^1$ the assertion $\t w_\el^\Le=0$ follows.\\
This demonstrates \er{3.21} in all cases, and concludes the proof of Lemma \ref{3.j}.
\end{proof}

We now complete the proof of Theorem \ref{3.b}. We need only to prove the sufficiency of the condition for F. Suppose that $F$ satisfies \er{3.28}. By Lemma \ref{3.j} there exists a constant $\lt\in\Tl$ such that $\dl^0_{v,v}F=\lt G''(0)(v)$ for all $v\in Z^1,$ where $G=G_c$ is given by \er{2.4}. By assumption, $F''(0)(e,z)=0$ for all $z\in Z^2\op Z^1.$ Therefore $F$ agrees up to second order with a Mok type embedding $G_\lt(z)=\o\lt(G(\lt z)).$ The following Lemma \ref{3.c} shows $F=G_\lt.$

\begin{lemma}\label{3.c} Let $\Bl$ be the unit ball in $\Cl^d$ and let $D$ be the symmetric domain in a Jordan triple $Z$ of rank $r>1,$
with $\dim(Z_c^2+Z_c^1)=d$ for a minimal tripotent $c\in Z$. Let $F:\Bl\to D$ be a Bergman isometry which agrees up to order 2 with a Mok embedding $G.$ Then $F=G.$
\end{lemma}
\begin{proof} Assume, by induction, that $F$ and $G$ agree up to order $n\ge 2.$ Then $F''(0)=G''(0)$ and Lemma \ref{3.d} implies for 
$x,u,v_0,\ldots,v_n\in V$ 
$$\<\dl^0_{x,u}G|\dl^0_{v_0,\ldots,v_n}F\>=\<\dl^0_{x,u}F|\dl^0_{v_0,\ldots,v_n}F\>=\<\dl^0_{x,u}G|\dl^0_{v_0,\ldots,v_n}G\>.$$
Now $\dl^0_{v_0,\ldots,v_n}F$ and $\dl^0_{v_0,\ldots,v_n}G$ take values in $Z^0,$ and the Jordan theoretic construction \er{2.1} implies that $Z^0$ is spanned by the vectors $\dl^0_{x,u}G,$ for $x,u\in Z^1$ arbitrary. Hence $\dl^0_{v_0,\ldots,v_n}F=\dl^0_{v_0,\ldots,v_n}G.$ Thus $F$ and $G$ agree up to order $n+1.$
\end{proof}

Specializing Theorem \ref{3.b} to the case of tube domains $D$ of rank 2, i.e. the Lie balls, and using Lemma \ref{3.e} we obtain:
\begin{corollary}\label{3.k} Suppose $F:\Bl_d\to D$ is a Bergman isometry from $\Bl_d$ into the Lie ball $D$ in $\Cl^{d+1}.$ Then $F$ is a Mok embedding if and only if $(Ran\ F'(0))^\perp=\Cl\xi$ for a minimal tripotent $\xi$.
\end{corollary}

\section{Irrational embeddings into the Lie ball}
The symmetric domains of tube type and rank 2 are the Lie balls, Example \ref{1.b}. In this section we construct and classify all isometric holomorphic embeddings from the unit ball $\Bl=\Bl_d$ into the Lie ball $D$ with $\dim D=d+1.$ Thus $\Bl$ and $D$ have the same genus $p=d+1.$ Similar results have also been obtained in \cite{XY} using explicit computations. Let $F:\Bl\to D$ be an isometric embedding into the Lie ball, with $F(0)=0.$

\begin{lemma}\label{4.a} Suppose $\dl^0_{e,e}F\ne 0$ for some $e\in\Cl^d.$ Then $\dl^0_eF$ is a scalar multiple of a maximal tripotent.
\end{lemma}
\begin{proof} By rescaling we may assume that $\|\dl^0_eF\|=\|\dl^0_{e,e}F\|.$ Put $\lh:=\dl^0_eF.$ Since $F'(0)$ is isometric and, by \er{1.10}, $Q_\lh\lh=\<\lh|\lh\>\lh-\f12\<\lh|\o\lh\>\o\lh$, we obtain with \er{3.5}
$$\<\lh|\lh\>^2=\<e|e\>^2=\<Q_ee|e\>=\f12\|\dl^0_{e,e}F\|^2+\<Q_\lh\lh|\lh\>$$
$$=\f12\|\lh\|^2+\<Q_\lh\lh|\lh\>=\f12\|\lh\|^2+\<\lh|\lh\>^2-\f12|\<\lh|\o \lh\>|^2.$$
Therefore $\<\lh|\lh\>=|\<\lh|\o \lh\>|,$ and Cauchy-Schwarz implies $\o\lh=\lt\lh$ for some $\lt\in\Tl.$ 
There exists a minimal tripotent $e_1$ such that $\lh=\la e_1+\lb\o e_1,$ where $\la>0$ and $\lb\in\Cl.$ Then
$$\o\lh=\la\o e_1+\o\lb e_1=\lt\lh=\lt\la e_1+\lt\lb\o e_1.$$
By uniqueness in the frame $e_1,\o e_1$ (or linear independence) it follows that $\o\lb=\lt\la.$ Thus $\lh=\la(e_1+\o\lt\o e_1)$
is a multiple of the maximal tripotent $e_1+\o\lt\o e_1.$
\end{proof}

\begin{theorem} Let $F:\Bl\to D$ be a Bergman isometric embedding into the Lie ball, with $F(0)=0.$ Write $Ran\ F'(0)^\perp=\Cl\lx$ for a unit vector $\lx.$ If $\lx$ has rank 1 (and hence is a minimal tripotent) then $F$ is a Mok embedding.
If $\lx$ has rank 2 (the non-rational case), then there exist a frame of minimal tripotents $e_1,e_2\in Z$ and a unimodular constant 
$\lt\in\Tl$ such that $F_\lt(z):=\o\lt F(\lt z)$ has the form
\be{4.3}F_\lt(u,v)=u\f{e_1+e_2}{\F 2}+v+\f{e_2-e_1}{\F 2}\(1-\F{1+u^2+\<\o e_1|e_2\>\ \<v|\o v\>}\).\ee
Conversely, any such mapping is an isometric embedding.
\end{theorem}

\begin{proof} By Lemma \ref{3.f}, we have
\be{4.1}F(z)=F'(0)z+h(z)\lx,\ee
where $h:\Bl\to\Cl$ is holomorphic with $h'(0)=0.$ For all $z,w\in\Bl$ it follows from \er{1.12} that
$$1-\<z|w\>=\lD(F(z),F(w))=1-\<F(z)|F(w)\>+N(F(z))\ \o{N(F(w))}$$
$$=1-\<F'(0)z+h(z)\lx|F'(0)w+h(w)\lx\>+N(F(z))\ \o{N(F(w))}$$
$$=1-\<z|w\>-h(z)\o{h(w)}+N(F(z))\ \o{N(F(w))},$$
since $F'(0)$ is isometric and $\lx$ is perpendicular to $Ran\ F'(0).$ Hence
$$h(z)\ \o{h(w)}=N(F(z))\ \o{N(F(w))}$$
for all $z,w\in\Bl$. Thus there exists a constant $\lt\in\Tl$ independent of $z$ such that $N(F(z))=\lt\cdot h(z).$ Replacing $F$ by 
$F_\lz(z):=\o\lt F(\lt z)$ we may therefore assume that $h(z)=N(F(z)),$ since $N$ is a quadratic polynomial.
If $\lx$ is a minimal tripotent, then the claim is proved in Corollary \ref{3.k}. Now suppose $\lx$ has rank 2. Then there exists a minimal tripotent $c$ and constants $\la>0,\lo\in\Cl$ such that $\lx=\la(\o\lo c-\o c)$.
 Then $\lh:=\la(c+\lo\o c)\in\lx^\perp=Ran\ F'(0)$ has also rank 2. Writing $\lh=\dl^0_eF$ for some $e\in\Cl^d$ it follows from Lemma \ref{3.e} that $\dl^0_{e,e}F\ne 0.$ By Lemma 
\ref{4.a} $\lh$ is a scalar multiple of a maximal tripotent, which is only possible if $|\lo|=1.$ Since $\lx$ is a unit vector, it follows that $\la=\f1{\F 2}$ and hence
\be{4.2}\lx=\f{\o\lo c-\o c}{\F 2}\in\ Ran\ F'(0)^\perp,\ \lh=\f{c+\lo\o c}{\F 2}\in\ Ran\ F'(0).\ee
Write $z\in \Cl^d$ as $z=ue+v$ with $u\in\Cl,v\in e^\perp.$
Identifying $Z^1:=<c,\o c>^\perp=(Ran\ F'(0))\ui\lh^\perp=F'(0)e^\perp$ with $e^\perp$ via $F'(0)$ we have $F'(0)z=u\lh+v.$ Since $\o\lx=-\lo\lx$ and $\o\lh=\o\lo\lh$ we obtain
$$2h(z)=2N(F'(0)z+h(z)\lx)=\<F'(0)z+h(z)\lx|\o{F'(0)z}+\o{h(z)}\o\lx\>=\lo u^2+\<v|\o v\>-\o\lo h(z)^2.$$
This yields the quadratic equation $h(z)^2+2\lo h(z)-\lo^2 u^2-\lo\<v|\o v\>=0,$ with solutions
$$h(z)=-\lo\pm\lo\F{1+u^2+\o\lo\<v|\o v\>}.$$
With \er{4.2} it follows that
$$F(ue+v)=u\f{c+\lo\o c}{\F 2}+v+\f{\lo\o c-c}{\F 2}(1\pm\F{1+u^2+\o\lo\<v|\o v\>}).$$
Now $e_1=c,e_2=\lo\o c$ runs over all frames of minimal tripotents, and $\o\lo=\<\o e_1|e_2\>.$ This yields \er{4.3}, the choice
of $\pm$ is uniquely determined by $F(0)=0$. Conversely it follows from the construction that $F_\lt(z)$ is locally defined near 
$z=(u,v)=0$ and satisfies $\lD(F(z),F(w))=1-\<z|w\>$. Thus $F_\lt$ is a local Bergman isometry which extends to an isometry from $\Bl$ to $D$ by Theorem 5.1 below.
\end{proof} 

\begin{remark} In \cite{Mok-ALM} Mok raised a series of questions on the rigidity of holomorphic isometries. The local rigidity at a Bergman holomorphic isometry $F$ is defined there by the following: For any relatively compact open set $\OO\ne U\ic\Bl$
there exists a $\delta>0$ such that any Bergman holomorphic isometry $h$ in a $\delta$-neighborhood of $F$, $\Vert F-h\Vert_{L^\infty(U)}\le\delta$, must be a reparametrization of $F.$ We remark here that the local rigidity does not hold for Mok's rational mapping 
$F_0$. Fix a frame of minimal tripotents $\{c_1,c_2\}$ and let $Z=\Cl c_1\op\Cl c_2\oplus Z^1$ and $\Bl$ be the unit ball in $\Cl c_1\op Z^1$. The mapping $F_0:\Bl\to D$ is then given by
$$F_0(z)=uc_1+\f{\<v|\o v\>}{1+u}c_2+v,\quad z=uc_1+v\in\Bl.$$
Let $\xi=t c_1+s c_2,\lh=s c_1-t c_2$ for $0<t,s<1,\ t^2 +s^2=1.$ The irrational mapping
$$F_{t}(u,v)=u\lh\op p_t(z)\xi\op v,\quad z=uc_1+v\in\Bl$$
with
$$p_t(z)=\f{-[(s^2-t^2)u-1]-\F{[(s^2-t^2)u-1]^2+4ts(ts u^2-\<v|\o v\>)}}{2ts}$$ 
is a Bergman isometry. This is a reparametrization of the mapping $F_\lt(z)$ in Theorem 4.1; its isometric property can also
be checked independently by direct computations. An easy computation shows that when $t\to 0$, $F_{t}\to F_0$ uniformly in any compact subset of the unit ball. Thus the local rigidity at $F_0$ does not hold.
\end{remark}

\section{Extension of holomorphic isometries}
In this final section we give alternative proofs of two results of Mok \cite{Mok-JEMS}, using basic ideas from Hilbert space operator theory. The first result is
\begin{theorem}\cite{Mok-JEMS} Let $f:(\Bl,\ll ds^2_B;0)\to(D,ds^2_D;0)$ be a germ of a holomorphic isometry into a bounded symmetric domain $D,$ with $f(0)=0.$ Then there exists a proper holomorphic isometric embedding $F:\Bl\to D$ extending the germ of holomorphic map $f$. 
\end{theorem}
\begin{proof}
We suppose $\ll=1$ and the general case is almost the same.
Let $\Omega$ be a bounded domain in $\mathbb C^d$ and  $H$  a Hilbert space consisting of holomorphic functions on $\lO$ which contains all polynomials. A point $z\in\Cl^d$ is called an {\bf evaluation point} of $H$ if and only if the map
$$\Lt_z: p\to p(z)\quad\forall\mbox{polynomials }p$$
is bounded on $H$. Let $vp(H)\ic\Cl^d$ be the set of all evaluation points. For any bounded symmetric domain $\lO$, we have 
$vp(L^2_a(\lO))=\lO$, which can be proved by some elementary arguments which we omit. For any choice of distinct points 
$z_1,\ldots,z_m\in\lO$, for $\lO=\Bl$ or $D$, the reproducing kernel matrix $A:=(K^\lO(z_i,z_j))$ is positive definite, and hence invertible. This implies that the kernel vectors $K^\lO_{z_j},\,j=1,\ldots,m$ are linearly independent. In fact, if $\S_{j}K^\lO_{z_j}v_j=0$ for some column vector $v\in\Cl^{m\xx 1},$ then for any $i$ we have
$$0=\S_{j}K^\lO_{z_j}(z_i)v_j=\S_{j}K^\lO(z_i,z_j)v_j,$$
i.e., $Av=0$ and hence $v=0.$ \\
{\bf Step 1: Isometric embedding of function spaces}. As argued in \cite{Mok-JEMS}, for some neighborhood $B_\ld$ of $0$ we have
$$K^\Bl(z_1,z_2)=K^D(f(z_1),f(z_2))$$
for $z,w\in B_\ld.$ Define a map $U:L^2_a(\Bl)\to L^2_a(D)$ by 
$$U(K^\Bl_z)=K^D_{f(z)}$$ 
for any $z\in B_\ld.$ Since the $K_z^\Bl$ are linearly independent, the map $U$ is well defined on the linear span $E$ of 
$\{K^\Bl_z, \,z\in B_\ld\}$. Then
$$\langle
U K^\Bl_z|U K^\Bl_w\rangle=\langle
K^D_{f(z)}|K^D_{f(w)}\rangle
=K^D(f(z),f(w))=K^\Bl(z,w)=\langle K^\Bl_z|K^\Bl_w\rangle,$$
for $z,w\in B_\ld,$ thus $U$ is isometric on $E.$ Since $E$ is dense
 in $L^2_a(\Bl)$  it follows that $U$ has a unique extension to an isometry $U$ on $L^2_a(\Bl).$ Moreover, we have 
$$U^*K^D_{f(w)}=K^\Bl_w$$ 
for all $w\in B_\ld$, since for any $z\in B_\ld$
$$\langle K^\Bl_z|U^*K^D_{f(w)}\rangle
=\langle U K^\Bl_z|K^D_{f(w)}\rangle=
\langle K^D_{f(z)}|K^D_{f(w)}\rangle=
K^D(f(z),f(w))=K^\Bl(z,w)=\langle K^\Bl_z|K^\Bl_w\rangle.$$\\
{\bf Step 2: Operators on quotient spaces}. Let $H=U L^2_a(\Bl)=\o{span}\{K^D_{f(z)}:z\in B_\ld\}\ic L^2_a(D)$. For any $z\in D$ the multiplication operator $M_g$ by bounded holomorphic functions $g$ on $D$ satisfies
$$M_g^*K^D_z=\o{g(z)}K^D_z$$
for $g\in H^\oo(D)$. Hence $H$ is invariant under $M_g^*$ and is therefore a quotient module in $L^2_a(D)$. Let $P$ be the orthogonal
projection from $L^2_a(D)$ onto $H$. Denote 
$$S_g=P M_g P.$$
Then $S_g^*=P M_g^*P=M_g^*P.$ For any polynomial $p$ and $z\in B_\delta$ we have
$$U^*S^*_g U M^*_p K^\Bl_z=\o{p(z)}U^*S^*_g U K^\Bl_z=\o{p(z)}U^*S^*_g K^D_{f(z)}$$
$$=\o{p(z)}U^*M^*_g K^D_{f(z)}=\o{p(z)g(f(z))}U^*K^D_{f(z)}=\o{p(z)g(f(z))}K^\Bl_z$$
and, similarly,
$$M_p^*U^*S^*_g U K^\Bl_z=M_p^*U^*S^*_g K^D_{f(z)}=M_p^*U^*M^*_g K^D_{f(z)}$$
$$=\o{g(f(z))}M_p^*U^*K^D_{f(z)}=\o{g(f(z))}M_p^*K^\Bl_z=\o{g(f(z))}\o{p(z)}K^\Bl_z.$$
By density, it follows that
$$(U^*S^*_gU)M^*_p=M^*_p(U^*S^*_gU).$$
Therefore $U^*S_gU$ commutes with each $M_p.$ Since $p$ is arbitrary, it follows that $U^*S_g U$ is a multiplier on $L^2_a(\Bl).$\\
{\bf Step 3: Map extension}. Choose the coordinate functions $w_1,\ldots,w_m$ on $D\ic\Cl^m$. Then $U^*S_{w_i}U$ is a multiplier on 
$L^2_a(\Bl)$. Therefore 
$$U^*S_{w_i}U=M_{F_i}$$ 
for some $F_i\in H^\oo(\Bl)$. Consequently $U^*S_q U=M_{q\,\oc F}$ for any polynomial $q$ on $D\ic\Cl^m$. We claim that 
$F=(F_1,\ldots,F_m)$ extends the germ $f=(f_1,\ldots,f_m)$. Indeed, for each $i$ and $z\in B_\ld$, we have
$$\o{F_i(z)}K^\Bl_z=M^*_{F_i}K^\Bl_z=U^*S^*_{w_i}U K^\Bl_z=U^*S^*_{w_i}K^D_{f(z)}=U^*M^*_{w_i}K^D_{f(z)}=\o{f_i(z)}K^\Bl_z.$$
Thus $F_i(z)=f_i(z)\, \forall i$.\\
{\bf Step 4: $F$ maps $\Bl$ into $D$}. For any $z\in\Bl$, we have to show that the map $q\to q(F(z))$ is bounded on $L_a^2(D).$ Since 
$K^\Bl_0=1,$ we have $UK_0^\Bl=K_0^D=1$ and
$$q\oc F=M_{q\oc F}(1)=U^*S_q U(1)=U^*Pq.$$ 
Therefore, when $z\in\Bl$,
$$|q(F(z))|\le C\|q\oc F\|_\Bl=C\|U^*Pq\|_\Bl\le C\|q\|_D$$
for some positive constant $C$. This implies $F(z)\in D$.\\
{\bf Step 5: $F$ is proper}. A reproducing kernel argument shows that $U K^\Bl_z=K^D_{F(z)}$ for each $z\in\Bl,$ and 
$K^\Bl(z_1,z_2)=K^D(F(z_1),F(z_2))$ for any $z_1,z_2\in\Bl.$ So, if $z\to\dl\Bl,$ then $\|K^\Bl_z\|\to\oo$ and hence $\|K^D_{F(z)}\|\to\oo$, $F(z)\to\dl D.$
\end{proof}

\begin{remark} We may also give an explicit formula for the isometric embedding $U$ in the proof of Theorem 5.1. We have
$$(Uh)(w)=\I_\Bl dz\,h(z) K^D_{F(z)}(w),\quad h\in L_a^2(\Bl),\quad w\in D.$$
To prove the identity it is enough to take $h=K^\Bl_z$, in which case $Uh(w)=K^D_{F(z)}(w)$ and 
$$\I_\Bl dt\,h(t)K^D_{F(t)}(w)=\I_\Bl dt\,K^\Bl_z(t)\ K^D_{F(t)}(w)=\o{K^D_w(F(z))}=K^D_{F(z)}(w)=(Uh)(w).$$
\end{remark}

\begin{remark} The steps 4-5 above can also be replaced by the following elementary argument, namely if $F:\Bl\to Z$ is a holomorphic map such that $F(0)=0$ and $K^D(F(z),F(z))=K^\Bl(z,z)$ in a neighborhood of $0\in\Bl$ then $F$ maps $\Bl$ isometrically into $D.$
Indeed, the equality $\lD(f(z),f(z))^{-p}=K^D(f(z),f(z))=K^\Bl(z,z)=(1-\<z|w\>)^{-p}$ is valid in a neighborhood of $z=0\in\Bl$, thus 
$\lD(f(z),f(z))=1-\<z|z\>$ since $p_\Bl=p_D=p.$ Now both sides are real analytic functions in $z$ and thus the equality holds for all 
$z\in\Bl$. Suppose for some $z_0\in\Bl$, $f(z_0)\notin D$. Then $\Vert f(z_0)\Vert_{Z}\ge 1$, with $\Vert\cdot\Vert_{Z}$ being the Jordan norm in $Z$. Consider the line segment $tz_0, t\in [0,1]$ in $\Bl$ from $0$ to $z_0$. Then $f(tz_0)$ is a curve from $0$ to 
$f(z_0)$. Thus for some $0<t_0\le 1$,
$$\Vert f(t_0z_0)\Vert=1,$$
namely $w_0=f(t_0z_0)$ is a boundary point of $D$. Any boundary point has its Peirce decomposition of the form
$$f(t_0z_0)=w_0=s_1c_1+\cdots+s_rc_r$$
with $\{c_1, \cdots, c_r\}$ being a frame of minimal tripotents and $s_1=1$ and $0\le s_j\le 1,\, 2\le j\le r$. Thus $\lD(w_0,w_0)=(1-s_1^2)\cdots (1-s_r^2)=0$, which implies in turn $1-\langle t_0z_0|t_0z_0\rangle=\lD(f(t_0z_0),f(t_0z_0))=\lD(w_0,w_0)=0$. But $tz_0$ is a point in $\Bl$ and $1-\<t_0z_0|t_0z_0\> >0$, a contradiction. The isometry of $f$ now follows from the definition of the Bergman metric.
\end{remark}

As another result in \cite{Mok-JEMS} Mok proves that for a holomorphic isometry $F:\Bl\to D$ there exists a complex-algebraic subvariety $\VL\ic\Cl^d\xx Z$ of dimension $d$ containing the graph $\GL_F=\{(z,F(z)):z\in\Bl\}.$ The Jordan theoretic approach makes this quite explicit and also shows equality $\VL\ui(\Bl\xx D)=\GL_F$. 

\begin{theorem} Let $F$ be a holomorphic Bergman isometry from a unit ball $\Bl\ic\Cl^d$ into a bounded symmetric domain $D\ic Z$. Then 
there exists a complex-algebraic subvariety $\VL\ic\Cl^d\xx Z$ such that
$$\GL_F:=\{(z,F(z)):z\in\Bl\}=\VL\ui(\Bl\xx D).$$
\end{theorem}
\begin{proof} For any $w\in Z,$ put $E_w^\k(z):=E^\k(z,w).$ Then, for the Fischer-Fock product conjugate linear in the first variable, we have $\<E_z^\k|E_w^\k\>=E^\k(z,w).$ Define polynomials
$$E_w^+:=\S_{0\ne k\ \tiny{even}}C_k^{1/2}\ E^\k_w,\qquad E_w^-:=\S_{k\ \tiny{odd}}C_k^{1/2}\ E^\k_w$$
in $W:=\S_{k=0}^{r}\PL_\k(Z)$. Then \er{1.7} implies for the Fock inner product
$$\lD(w_1,w_2)=1+\<E_{w_1}^+|E_{w_2}^+\>-\<E_{w_1}^-|E_{w_2}^-\>$$
for all $w_1,w_2\in Z.$ Therefore the isometry property $\lD(F(z_1),F(z_2))=1-\<z_1|z_2\>$ is equivalent to the identity
\be{5.1}\<E_{F(z_1)}^-|E_{F(z_2)}^-\>=\<E_{F(z_1)}^+|E_{F(z_2)}^+\>+\<z_1|z_2\>\ee
for all $z_1,z_2\in B.$ This implies that 
$$U E_{F(z)}^-:=E_{F(z)}^+\op z$$
defines an isometry from the linear span $\{E_{F(z)}^-:\ z\in\Bl\}\ic W$ into $W\op\Cl^d$ (orthogonal sum). Let $U:W\op\Cl^d\to W\op\Cl^d$ 
be any unitary extension. Then
\be{5.2}\VL:=\{(z,w)\in\Cl^d\xx Z: U E_w^-=E_w^+\op z\}\ee
is an algebraic variety which contains the graph $\GL_F$ by construction. Conversely, let $(z_0,w_0)\in\VL\ui(\Bl\xx D).$ Observe
that for any $(z,w)\in\VL\ui(\Bl\xx D)$
\begin{equation*}\begin{split}\<E_{w_0}^-|E_{w}^-\>&=\<UE_{w_0}^-|UE_{w}^-\>=\<E_{w_0}^+\op z_0|E_{w}^+\op z\>\\
&=\<E_{w_0}^+|E_w^+\>+\<z_0|z\>\end{split}\end{equation*}
since $U$ is an isometry. In particular, 
$$\lD(w_0,w_0)=1+\<E_{w_0}^+|E_{w_0}^+\>-\<E_{w_0}^-|E_{w_0}^-\>=1-\<z_0|z_0\>$$
and 
$$\<E_{w_0}^-|E_{F(z)}^-\>=\<E_{w_0}^+|E_{F(z)}^+\>+\<z_0|z\>$$
for every $z\in\Bl.$ Therefore
$$\lD(w_0,F(z))=1-\<E_{w_0}^+|E_{F(z)}^+\>+\<E_{w_0}^-|E_{F(z)}^-\>=1-\<z_0|z\>$$
and hence $K^D(w_0,F(z))=K^B(z_0,z).$ Fixing $z=z_0$, we have in the Bergman space
$$\langle
K^D_{F(z_0)}|K^D_{w_0}-K^D_{F(z_0)}
\rangle=K^D(w_0,F(z_0))-K^D(F(z_0),F(z_0))=K^\Bl(z_0,z_0)-K^\Bl(z_0,z_0)=0.$$
Therefore, $K^D_{F(z_0)}\perp K^D_{w_0}-K^D_{F(z_0)}$ and
\begin{equation*}\begin{split}\|K^D_{w_0}-K^D_{F(z_0)}\|^2&=\|K^D_{w_0}\|^2-\|K^D_{F(z_0)}\|^2=\lD(w_0, w_0)^{-p}-\lD(F(z_0), F(z_0))^{-p}\\&=(1-\<z_0|z_0\>)^{-p}-(1-\<z_0|z_0\>)^{-p}=0.\end{split}\end{equation*}
This means $w_0=F(z_0)$ and $(z_0,w_0)\in\GL_F$, completing the proof.
\end{proof}

\begin{remark}
Consider the more general case of a 'scaled' Bergman isometry $F:\Bl\to D,$ satisfying the condition $F(0)=0$ and
\be{5.3}1-\<z|w\>=\lD(F(z),F(w))^\la\ee
for some parameter $\la>0.$ It has been shown by Mok that $\la$ is a rational number $\la=\f lk$, i.e. $(1-\<z|w\>)^k
=\lD(F(z),F(w))^l$, and hence the same proof as above works by expanding $\lD(F(z),F(w))^l$ according
to \er{1.6}, $(1-\<z|w\>)^k=\S_{j=0}^k (-1)^j\binom k j\<z|w\>^j$, and obtaining a similar equality as \er{5.1}.
\end{remark}

\end{document}